\documentclass[12pt]{amsart}
\usepackage[margin=1in,marginparwidth=0.8in, marginparsep=0.1in]{geometry}

\usepackage{amsfonts,amssymb,latexsym,amsmath,graphicx, amscd}

\usepackage[
bookmarks=true, bookmarksopen=true,%
bookmarksdepth=3,bookmarksopenlevel=2,%
colorlinks=true,%
linkcolor=blue,%
citecolor=blue,%
filecolor=blue,%
menucolor=blue,%
urlcolor=blue]{hyperref}

\usepackage[all, knot, poly]{xy}
\xyoption{knot}

\usepackage{times}
\usepackage{mathrsfs}

\newtheorem{lemma}{Lemma}
\newtheorem{proposition}[lemma]{Proposition}
\newtheorem{corollary}[lemma]{Corollary}
\newtheorem{theorem}[lemma]{Theorem}

\theoremstyle{definition}

\theoremstyle{remark} 
\newtheorem*{remark}{Remark}

\newcommand{\R}{\mathbb{R}}

\newcommand{\T}{\mathbb{T}}

\newcommand{\Z}{\mathbb{Z}}

\newcommand{\cF}{\mathcal{F}}
\newcommand{\cG}{\mathcal{G}}
\newcommand{\cH}{\mathcal{H}}

\newcommand{\cL}{\mathcal{L}}

\newcommand{\cN}{\mathcal{N}}

\DeclareMathOperator{\Hom}{Hom}

\author{Vivek Shende}

\title{The conormal torus is a complete knot invariant}

\begin{document}

\begin{abstract}
We use microlocal sheaf theory to show: knots can only have Legendrian isotopic conormal tori if they themselves are isotopic or mirror
images.  
\end{abstract}

\maketitle

\section{Introduction} 

A knot $K \subset \R^3$ determines, by taking the unit conormals, a 
torus in the unit cosphere bundle.
By general position arguments, 
the smooth type of this embedding $\T_K \subset S^* \R^3$ 
knows nothing about the knot.  However, this embedding is that
of a Legendrian in a contact manifold, and a smooth isotopy of 
knots $K \sim K'$ induces a Legendrian isotopy of conormals
$\T_K \sim \T_{K'}$.  
Thus the Legendrian isotopy type of  $\T_K$ is a topological invariant of $K$. 

More generally, taking the cosphere or cotangent
bundle gives a natural functor from smooth topology to 
contact or symplectic geometry, and it is a fundamental 
question in these subjects to understand what this functor 
remembers and what it forgets.  
A representative example 
is Arnold's conjecture that every 
compact exact Lagrangian in the cotangent bundle 
Hamiltonian isotopic to the zero section; there has been 
much recent work in this direction \cite{N1, FSS, Abo, AK}.

The question of the extent to which $\T_K$ determines $K$ 
has previously been studied by 
holomorphic curve techniques, i.e., by computing the relative
contact homology of the pair $(S^* \R^3, \T_K)$.  Perhaps
the first indication of the true strength of this invariant was 
in \cite{Ng3}, where it was shown to detect the 
unknot.  More recently, this was extended to  torus knots \cite{GLi}.  

In the present article we take a different approach to the study 
of $\T_K$.  In general, one can study the geometry of Legendrians in cosphere bundles
by assigning, to a Legendrian $\Lambda \subset S^* M$, 
the category of sheaves with microsupport in $\Lambda$. 
This approach reaches back in some sense to the microlocal analysis
of Sato and H\"ormander, is technically 
built on the microlocal sheaf theory of Kashiwara and Schapira \cite{KS}, 
and was applied to study symplectic geometry by Tamarkin \cite{T},
and subsequent authors \cite{GKS, Gui, Gui2, Gui3, Chi, STZ, STWZ, STW}.  
Using this method, we  give here the first proofs of the following
results: 

\begin{theorem}
Let $L, L'$ be oriented non-split links in $\R^3$.  If there is a parameterized 
Legendrian isotopy $\T_{L} \to \T_{L'}$, then there is a 
topological isotopy $L \to L'$. 
\end{theorem}

\begin{theorem}
For knots $K, K'$ in $\R^3$,  if there is any
Legendrian isotopy $\T_K \to \T_{K'}$, then the knots are either isotopic,
or mirror. 
\end{theorem}

In the first theorem above, by ``parameterized isotopy'' 
we mean that the isotopy respects
the homotopy classes of the meridians and longitudes of the tori; we 
{\em do not} impose this condition in the second theorem.  Also, 
in the first theorem, by non-split, we mean that no sphere separates
some components of the link from some others.

The  strategy of proof is as follows.  A standard consequence of 
Gray's stability theorem \cite{Gr} is that a Legendrian isotopy
is induced by an ambient contact isotopy, which can moreover be chosen
trivial away from an open set around the Legendrian isotopy. Henceforth by Legendrian
isotopy we will always implicitly mean such an extension. 
The sheaf quantization of this contact isotopy \cite{GKS} gives an equivalence of derived categories 
$shv_{\T_K}(\R^3) \cong shv_{\T_{K'}}(\R^3)$.  

Whether or not such a {\em derived equivalence of knots} forces
the knots to be isotopic is, to our knowledge, open.  Rather than resolve that question, 
we try and rigidify the above isomorphism.  This is a nontrivial task: 
it is not generally possible to extract a ring from its abelian category
of modules, and it is not generally possible to extract an abelian
category from its derived category.  For instance,
the existence and significance
of derived equivalences in birational algebraic geometry \cite{Orl, BO}
is a subtle and fascinating subject.  Nonetheless, we are able to 
extract $\Z[\pi_1(\R^3 \setminus K)]$ from
the derived category of sheaves $shv_{\T_K}(\R^3)$; we also characterize 
the classes of the longitude and meridian.  
This is the main calculational work of this paper.  

What makes it possible, ultimately, is the fact that because we 
are working in the noncompact $\R^3$, the isotopy can be pushed off the 
fibre at a point.  This point provides a fibre functor which rigidifies our 
categories sufficiently to eventually extract the group ring.   
In particular, we do not know whether the result still holds with 
$\R^3$ replaced by $S^3$.  This would follow if it was shown 
that no derived equivalences of knots could exist.

To finish
the proof we appeal to known results in 3-manifold topology. 
First,  
the fundamental group of a link complement is left-orderable
 \cite[Lemma 2]{HS}, hence can be recovered from its
integral group ring.  Finally, by the work of Waldhausen \cite{Wal},
the fundamental group of the complement 
together with the classes of the meridians
and longitudes determines the link.

\vspace{3mm}

\subsection*{Acknowledgements} This project was inspired by discussions with Christopher Cornwell and Lenhard Ng,
in the wonderful environment of the Mittag-Leffler institute.  I also thank Laura Starkston and Jacob
Tsimerman for helpful conversations, Harold Williams for repeated explanations regarding the difference
between finite and infinite rank local systems, and Dusa McDuff for encouragement to write down
this result.  I am partially supported by NSF DMS-1406871.

\section{Recovering the fundamental group}

We recall the language of microlocal sheaf theory; for details see \cite{KS}.   
To a sheaf $\cF$ on $M$, there is a locus
$ss(\cF) \subset S^* M$ --- the directions along which the sections of the sheaf fail to propagate. 
A Legendrian $\Lambda \subset S^* M$ determines a full subcategory of the derived category of 
sheaves, the objects of which are: 
$$shv_\Lambda(M) = \{ \cF \mbox{ a sheaf on } M\,|\, ss(\cF) \subset \Lambda\}$$
Guillermou, Kashiwara, and Schapira \cite{GKS} (see also Tamarkin \cite{T}) 
show this category is a contact invariant, 
in the sense that: 

\begin{theorem} \cite{GKS}
A compactly supported contact isotopy on $S^* M$ carrying $\Lambda \rightsquigarrow \Lambda'$ 
induces an equivalence  $shv_\Lambda(M) \cong shv_{\Lambda'}(M)$.  
\end{theorem}

The relation to the problem of interest is clear.  For a submanifold 
$Y \subset X$, we write $\T_Y \subset S^*X$ for the unit  conormal bundle.

\begin{corollary} Let $N, N'$ be compact submanifolds of $M$.  A Legendrian isotopy 
of conormal tori $\T_N \rightsquigarrow  \T_{N'} $ induces an equivalence of  
categories $shv_{\T_N}(M) \cong shv_{\T_{N'}}(M)$. 
\end{corollary}

That is, the category $shv_{\T_N}(M)$ depends only on the Legendrian isotopy
type of  $\T_N$.  

\vspace{2mm}
For the above results, it is essential to work with the derived 
category of sheaves (or a dg enhancement), rather than the abelian category.   However, 
derived categories tend to have many autoequivalences, making them somewhat slippery 
to use as invariants; consider for example the existence of nontrivial derived equivalences of algebraic varieties \cite{Orl}.  Thus we will seek to produce
more familiar invariants.

As we have allowed the microsupport to occupy the whole conormal bundle to $N$, the 
category $shv_{\T_N}(M)$ is just the category of sheaves constructible with respect
to the stratification $$M = (M \setminus N) \cup N$$ That is,  $\cF \in shv_{\T_N}(M)$
if and only if $\cF|_{M \setminus N}$ and $\cF|_{N}$ are locally constant.

%

%
%

Let $i: N \hookrightarrow M \hookleftarrow M \setminus N: j$ be the open and closed inclusions.  
The (derived) categories of local systems on $N$ and $M \setminus N$ embed fully faithfully into 
$shv_{\T_N}(M)$ via $i_! = i_*$ and $j_!$ respectively.  We will identify these categories of
local systems with their images, and denote them by $loc(N)$ and $loc(M \setminus N)$.  Any object
can be decomposed into objects from these subcategories, by the devissage exact triangle
(the analogue of excision in sheaf theory): 
$$j_! j^! \cF \to \cF \to i_* i^* \cF \xrightarrow{[1]}$$

From this, the adjunction $\Hom_M(\cF, i_* \cN) = \Hom_{N}(i^* \cF, \cN)$, and the  vanishing
$i^*j_! = 0$,  it follows that 
$loc(M \setminus N)$ is characterized as the ``right orthogonal complement'' of $loc(N)$, i.e.,
those objects whose homs to $loc(N)$ are zero.  Likewise $loc(N)$ is the left orthogonal complement
to $loc(M \setminus N)$.  

%
%

\begin{theorem}  \label{thm:groupring}
If $N$ is compact and $M$ is noncompact, then a Legendrian isotopy 
$\T_N \to \T_{N'}$ induces an isomorphism of integral group rings
$\Z[\pi_1(M \setminus N)] \cong \Z[\pi_1(M \setminus N')]$. 
\end{theorem}
\begin{proof}
Fix a compactly supported contact isotopy inducing the Legendrian isotopy,
and choose  $m\in M$ beyond the projection to $M$ of this compact set.  
Denoting by $\Phi$ the \cite{GKS} equivalence induced by the isotopy, we have a natural isomorphism
$\cF_m \cong (\Phi \cF)_m$.

The subcategories $loc(N)$ and $loc(N')$ of $shv_{\T_{N}}(M)$ and 
$shv_{\T_{N'} M}(M)$ can be characterized as the full subcategories of objects whose stalk at
$m$ vanishes.  Thus $\Phi$ carries $loc(N)$ to $loc(N')$.  Hence also $\Phi$ carries the right orthogonal
complement of one of these to the right orthogonal complement of the other, i.e., 
carries $loc(M \setminus N)$ to $loc(M \setminus N')$.  
The abelian, rather than derived, category of modules over $\Z[\pi_1(M \setminus N)]$ is characterized
as the subcategory of $loc(M \setminus N, \Z)$ whose fiber at $m$ has cohomology concentrated in degree
zero.  

Now recall that any ring $A$ can be recovered as the endomorphisms of the forgetful functor 
$A-mod \to \Z-mod$.  Thus we can recover $\Z[\pi_1(M \setminus N)]$ from the category 
$loc(M \setminus N, \Z)$ together with the fiber functor $\cL \to \cL_m$.  Strictly speaking, this
argument requires allowing infinite $\Z$-rank 
local systems, specifically the local system with fiber $\Z[\pi_1(M \setminus N)]$, but this is unproblematic.
%
\end{proof}

\begin{remark}
By {\em not} passing to the abelian category,
one could learn that in fact the rings of chains on the based loop spaces on 
$M \setminus N$ and $M \setminus N'$ are quasi-isomorphic. One can
learn even more by working over the sphere spectrum rather than $\Z$;
it is known to experts that the microlocal sheaf theory works in this generality,
although the details have not appeared in the literature. 
\end{remark}

It is not generally true that $\Z[G]$ determines $G$; there are counterexamples
even amongst finite groups \cite{Her}.
The most basic difficulty, studied by Higman in his thesis \cite{Hig}, is that $\Z[G]$ can have nontrivial units, i.e., 
those which are not $\pm g$ for
$g \in G$.  He conjectured (and it remains open) that a torsion free group has no nontrivial units; he also 
observed that imposing the much stronger condition of left-orderability ensures that there are 
no nontrivial units.  Note that such a group can be recovered from $\Z[G]$ as the quotient of the units by 
the torsion units, since as $G$ is torsion free and there are no nontrivial units, the only torsion units are $\pm 1$. 

\begin{remark}
In light of the previous remark, it is natural to ask whether 
Hertling's counterexample still works for the sphere-spectrum group-ring. 
\end{remark}

Fortunately, it is known that the fundamental group of a non-split
link complement is locally indicable, hence left-orderable \cite{HS}. 
We conclude: 

\begin{theorem}
If $L, L' \subset \R^3$ are links with Legendrian isotopic conormal tori, then 
their complements have isomorphic fundamental groups. 
\end{theorem}

\begin{remark}
One might want to try and deduce an isomorphism of knot groups instead by showing the equivalence of categories 
preserved the monoidal
structure.  However, this does not follow from the \cite{GKS} construction: 
the isomorphism is constructed
as a convolution, and convolution does not generally respect the tensor product.  
Moreover, along the isotopy, the Legendrian is not generally 
a conormal, and in this case, the category $shv_\Lambda(\R^3)$ is
not closed under tensor product. 
\end{remark}

From this it already follows that the conormal torus distinguishes amongst prime knots by 
the work of Waldhausen
\cite{Wal} and Gordon and Luecke \cite[Cor. 2.1]{GL}.  

\section{Longitudes and meridians}

To get the full result,
we should show that the isomorphism respects the ``peripheral subgroups'', i.e., 
the images of
the fundamental groups of the boundary tori of $\R^3 \setminus L$ and of $\R^3 \setminus L'$; note that having done so we need appeal only to \cite{Wal}
and not additionally \cite{GL}.  
It will suffice to show that there is some
method, {\em commuting with the sheaf quantization $\Phi$ above}, 
to extract the holonomies of a given local system around the longitude
and meridian of each component of the link.  Indeed, having done so, 
one considers the infinite rank local system given by the fundamental group
ring of the complement, and extracts the name of the meridian or longitude
from the fact that holonomy around said element is given by multiplication 
by said element.

We will use the microlocalization functor of \cite{KS}.
$$\mu \cH om : shv(X)^{op} \times shv(X) \to shv(T^* X) $$ 
The point is that, on the one hand, as we recall below, 
this microlocalization functor can recover the peripheral data; but on the other
{\em microlocalization commutes with sheaf quantization}. That is, 
given a contact isotopy $\phi: S^*X \to S^*X$, and its sheaf quantization
$\Phi: shv(X) \to shv(X)$, one has the compatibility \cite{GKS}: 

$$\phi_* \mu \cH om(\cF, \cG)|_{S^*X} = \mu \cH om(\Phi \cF, \Phi \cG)|_{S^*X}$$

We have shown above that $\Phi$ carries local systems to local systems,
and moreover respects the subcategory $Loc \subset loc$ of local systems
in the usual sense, as opposed to derived local systems.  By abuse
of notation we also write $\Phi$ for the corresponding restricted functor. 

\begin{proposition}
For a component $K \subset L$ and the corresponding
component $K' \subset L'$, 
the following diagram commutes, and the vertical maps carry
information equivalent to the restriction of the local system to the torus
which is the boundary of a tubular neighborhood.
$$
\begin{CD}
Loc(\R^3 \setminus L) @> \Phi >> Loc(\R^3 \setminus L')\\
@V\mu \cH om(\Z_K, j_! \cdot )|_{\T_K} VV  @VV\mu \cH om(\Z_{K'}, j_! \cdot )|_{\T_{K'}}V\\ 
Loc(\T_K) @> \phi_* >> Loc(\T_{K'}) \\
\end{CD}
$$
\end{proposition}

\begin{proof}
In the case at hand, one has by \cite[Prop. 4.4.3]{KS}
$$\mu \cH om(\Z_K, j_! \cL)|_{\T_K} = 
(\mu_K j_! \cL)|_{\T_K} = (\mathfrak{F}_* \nu_K j_! \cL)|_{\T_K}$$
Above, $\mu_K$ is the functor of microlocalization at $K$, which returns 
a sheaf on $T^*_K \R^3$.  It is, by definition, the Fourier-Sato transform
$\mathfrak{F}_*$ of the Verdier specialization $\nu_K$ at $K$. 
Verdier specialization, in general, is the rescaling limit of a sheaf under
the deformation to the normal bundle; in the case at hand, it can be obtained
just by pulling back along an identification with $T_K \R^3$ with a tubular 
neighborhood of $K$ in $\R^3$.  That is, having made this identification, 
and also writing $j: (T_K \R^3 \setminus 0) \hookrightarrow
T_K \R^3$, we have $\nu_k j_! \cL = j_! \cL$.  Taking the Fourier transform
and restricting to the unit torus $\T_K$ 
gives 
a local system with the same stalk as $\cL$ and monodromies
given by some fixed transformation from of those of $\cL$ (it will
not be important what this transformation is). 

The above argument would prove the statement of the
proposition with $\Z_{K'}$ replaced by 
$\Phi \Z_{K}$.  This is already good enough to show that the 
peripheral subgroup is carried to the peripheral subgroup, but 
to extract detailed information regarding longitudes and
meridians, we need to know that in fact $\Phi \Z_K = \Z_{K'}$.  

To prove this, note first $(\Phi \Z_K)_m = \Z_K|_m = 0$, hence 
$\Phi \Z_K$ is a local system on $K'$.  
Since
$$\Z_{\T_K} = \mu \cH om(\Z_K, \Z_K)|_{\T_K} = 
\phi^* (\mu \cH om(\Phi \Z_K, \Phi \Z_K)|_{\T_{K'}})$$
the sheaf $\Phi \Z_K$ must be concentrated in a single degree, and
rank one there. 

Since $ss(\Phi \Z_{\R^3}) = \phi ss(\Z_{\R^3}) = \emptyset$,
the sheaf $\Phi \Z_{\R^3}$ is locally constant, hence constant since $\R^3$ is
contractible.  Since
$(\Phi \Z_{\R^3})_m = \Z_{\R^3}|_m$ and both sheaves are constant, $\Phi \Z_{\R^3} = \Z_{\R^3}$.  

Thus we have 
$$H^*(S^1, \Z) = \Hom(\Z_{\R^3}, \Z_K) = 
\Hom(\Phi \Z_{\R^3}, \Phi \Z_K) = \Hom(\Z_{\R^3}, \Phi \Z_K) = 
H^*(S^1, \Phi \Z_K)$$

Since the only rank 
one local system on the circle with nontrivial cohomology is 
the constant local system, we conclude that $\Phi \Z_K = \Z_{K'}$. 
\end{proof}

To finish the proof of theorem 1, note that by hypothesis, 
$\phi$ carries meridians to meridians and longitudes to longitudes. 
By the above proposition, it therefore preserves 
holonomy of the microlocalizations around these, which 
is equivalent  (in the same way for both $K$ and $K'$) 
data to the holonomies of the local system around the boundary
of a tubular neighborhood. $\blacksquare$

\vspace{2mm}
We turn to the proof of Theorem 2.  What we need to show 
is that, even without assuming that the isotopy is parameterized, 
the longitude and meridian of $\T_K$ are carried {\em up to signs} 
to the longitude and meridian of $\T_K'$.  In fact we will 
show this for a possibly multiple-component link, but we can only 
conclude the stated result regarding being equivalent or mirror 
in the case of knots, as the signs may be different for each component
of the link. 

Let $K$ be some component of $L$, and $K'$ the corresponding 
component of $L'$.  An argument
similar to that showing $\Phi \Z_K = \Z_{K'}$ shows that $\Phi$
carries nontrivial local systems on $K$ to nontrivial local systems on $K'$.  
Let $\cL$ and $\Phi \cL = \cL'$ be such local systems.  
The meridian of $\T_K$ can be characterized, up to sign, as the primitive
class in $\pi_1(\T_K, \Z)$ for which the holonomy of $\mu \cH om(\Z_K, \cL)$ is trivial.  Note
$\phi_* \mu \cH om(\Z_K, \cL) = \mu \cH om(\Phi \Z_K, \Phi \cL) = 
\mu \cH om (\Z_{K'}, \cL')$. 
We conclude that $\phi$ carries the meridian of $K$ to the meridian
of $K$ up to sign.   Having found the meridians, the longitude of
$K$ can be characterized up to sign as the primitive class in 
$\pi_1(\T_K, \Z)$ which goes
to zero in the quotient of $H_1(\R^3 \setminus L, \Z)$ by all meridians. 
$\blacksquare$

\section{Epilogue}

In addition to being natural 
in pure mathematics, the question of the invertibility of 
the cotangent or cosphere functors has
a physical motivation: 
quantum mechanics concerns the analysis of functions on
$M$ whereas classical mechanics concerns 
the symplectic geometry of $T^*M$ --- how much does
dequantization remember?  
In fact, 
our approach to the study of $\T_M$ is directly related to this 
perspective.  More precisely, the sheaf category can be viewed as standing in 
for the pseudo-differential operators with characteristic variety
$\T_K$; the microlocal analysis of these operators 
has a precise relation to the WKB method
in quantum mechanics; see e.g. \cite{WB}. 

In the case of knots, there is 
also a rather different physical motivation. 
According to Witten, 
the Chern-Simons invariants
of 3-manifolds 
(aka Reshetikhin-Turaev invariants) 
in terms of topological string theory in the cotangent bundle \cite{Wit}.
There is a version for knots, 
in which the conormal serves as a Lagrangian or Legendrian
boundary condition \cite{Wit, OV, AV, AENV}.  
The fact that one can conjecturally extract such strong 
invariants from $\T_K$ was another motivation for
the question of the extent to which
it determines $K$.  

The holomorphic curve approach to this question is the 
mathematical interpretation of the above string theoretic notions. 
Here 
%
%
%
%
one associates a certain differential graded algebra 
to a Legendrians in a contact manifold.  
In the case at hand it can be effectively computed \cite{Ng1, Ng2}, 
and related to the 
fundamental group of the knot complement \cite{Ng3, CELN}. 
Thus, Legendrian isotopies between conormal tori induce
some relations between knot dgas and therefore between knot groups.

One can prove,
along the lines of \cite{NZ, N1, NRSSZ}, a  ``brane quantization is sheaf quantization''  
result asserting that a representation category
of knot contact homology is equivalent to the category of sheaves studied here.   
The idea is to associate to any augmentation of knot contact homology the sheaf on $\R^3$ whose 
stalk is given by the bilinearized contact homology \cite{BC, NRSSZ} from the cosphere over that point to the given
augmentation.  (More precisely, this captures the ``Q=1'' specialization of the knot contact homology.)  
The most difficult point of the argument is showing that this map is essentially surjective;
details will appear elsewhere \cite{ENS2}.  In particular, it will follow that a certain variant of the knot contact DGA is also
a complete knot invariant.
This construction can be shown to agree with the more ad-hoc constructions of Ng \cite{Ng1, Ng4} and Cornwell 
\cite{Cor1, Cor2} relating
augmentations to local systems and knot contact homology to $\pi_1$, when the latter are defined.

In terms of this equivalence, one can ask why the previous works \cite{Ng3, GLi} did not arrive at the result proven here. 
One reason is presumably that the sheaf category is by nature closer to local systems, so it
is easier to work with.  More precisely, the $j_! \cL$ used here do not sit in the heart of the natural $t$-structure for
the augmentation category.  But in fact the key point is 
 the stalk-at-$m$ functor.  
This can be introduced in the augmentation category by working relative to a cotangent fiber over $m$; we have since done so and given
a direct holomorphic-curve argument for the result proven here \cite{ENS} without
appealing to the (still forthcoming) \cite{ENS2}.  This moreover 
shows that a certain enhancement of the relative contact homology 
already is a complete knot invariant. 

More recently still, 
another work has appeared which  establishes the equivalence of the sheaf category with the
knot DGA \cite{BEY}.  This proceeds by appeal to the representation-theoretic formulation of \cite{Ng1} rather
than the geometric techniques of \cite{CELN}, and may be combined with the present article to give
another approach to the results of \cite{ENS2}.

%

\end{document}